\documentclass[12pt, reqno]{amsart}
\usepackage{amsmath}
\usepackage{amssymb}
\usepackage{epsfig}
\usepackage{graphicx}
\usepackage{color}
\usepackage{fullpage}
\usepackage{comment}
\definecolor{shadecolor}{gray}{0.875}
\usepackage{amscd}
\usepackage{stmaryrd}

\numberwithin{equation}{section}

\input xy
\xyoption{all}

\calclayout
\allowdisplaybreaks[3]

\theoremstyle{plain}
\newtheorem{prop}{Proposition}[section]

\newtheorem{theo}[prop]{Theorem}

\theoremstyle{definition}
\newtheorem{defi}[prop]{Definition}

\newtheorem{prob}[prop]{Problem}

\newtheorem{conj}[prop]{Conjecture}

\newtheorem{rema}[prop]{Remark}

\newtheorem{exam}[prop]{Example}

\def\bZ{{\mathbb Z}}

\def\Nef{\mathrm{Nef}}

\makeatother
\makeatletter

 \author{S\lowercase{ho} T\lowercase{animoto}}
\address{Graduate School of Mathematics, Nagoya University, Furocho Chikusa-ku, Nagoya, 464-8602, Japan}
\email{sho.tanimoto@math.nagoya-u.ac.jp}

\title[Geometric Manin's conjecture]{A\lowercase{n} \lowercase{introduction to} G\lowercase{eometric} M\lowercase{anin's conjecture}}

\begin{document}
\date{\today}

\begin{abstract}
This is a survey paper on Geometric Manin's conjecture which was proposed by Brian Lehmann and the author. We introduce Geometric Manin's conjecture (GMC) and review some recent progress on this conjecture.
\end{abstract}

\maketitle

\tableofcontents

\section{Introduction}

In the late 1980s Yuri Manin and his collaborators (Victor Batyrev and Yuri Tschinkel) formulated a conjecture predicting an asymptotic formula for the counting function of rational points of bounded height on a smooth Fano variety defined over a number field (\cite{FMT89} and \cite{BM}), and this has been further revised over three decades by many mathematicians including but not limited to Emmanuel Peyre, Batyrev--Tschinkel, and Lehmann--Sengupta--Tanimoto in \cite{Peyre}, \cite{BT}, \cite{Peyre03}, \cite{Peyre17}, and \cite{LST18}. This conjecture has been formulated and summarized as {\it Manin's conjecture} or {\it Batyrev--Manin's conjecture}. 

Using the correspondence between rational points over number fields and rational curves over finite fields, one may formulate a version of Manin's conjecture over finite fields or more generally over global function fields. When Manin's conjecture was formulated, Batyrev discovered a heuristic argument implying Manin's conjecture over finite fields (\cite{Bat88}), and this was an inspiration of \cite{BM} which includes the formulation of the first version of Manin's conjecture over number fields. Batyrev's heuristic was based on the following three assumptions on the moduli spaces of rational curves on Fano varieties: 
\begin{enumerate}
\item the moduli spaces parametrizing rational curves outside of a proper closed exceptional set have expected dimension; 
\item for each nef algebraic class $\alpha$ of integral $1$-cycles, the moduli space of rational curves of class $\alpha$ is irreducible, and;
\item the number of rational points on the moduli space is approximated by the number of rational points on the affine space of the same dimension.
\end{enumerate}
Furthermore, Jordan Ellenberg and Akshay Venkatesh suggested to justify (3) by using homological stability of \'etale cohomology of the moduli spaces combined with Grothendieck-Lefschetz trace formula. (\cite{EV05}) The assumption (1) has been verified over complex numbers in \cite[Theorem 1.1]{LTGMC} using results on exceptional sets in Manin's conjecture (\cite{HJ16} and \cite{LTDuke}) which are based on recent breakthroughs in the minimal model program (\cite{BCHM} and \cite{birkar16b}). The assumption (2) is wrong in general due to the presence of thin exceptional sets which are geometrically characterized in \cite{LST18}. To overcome this situation, Brian Lehmann and the author proposed Geometric Manin's conjecture in \cite{LTGMC} as a replacement of the assumption (2), and this has been further revised in \cite{BLRT20} by Roya Beheshti, Brian Lehmann, Eric Riedl, and the author.

Roughly speaking Geometric Manin's conjecture predicts that for each nef algebraic class $\alpha$ of sufficiently large anticanonical degree, there is a unique irreducible component of the moduli space parametrizing rational curves of class $\alpha$ which should be counted in Manin's conjecture. 
This unique component is called as the {\it Manin component} in \cite{BLRT20}.
In this survey paper, we review the formulation of Geometric Manin's conjecture developed in \cite{LTGMC} and \cite{BLRT20} and discuss some recent results on this topic.

\

Here is the road map of this paper: in Section~\ref{sec:background} we review some basic definitions and results which are needed in the later sections. In Section~\ref{sec:ManinBatyrev}, we recall Manin's conjecture over finite fields and Batyrev's heuristic of this conjecture. In Section~\ref{sec:GMCchar0}, we introduce Geometric Manin's conjecture in characteristic $0$ and discuss some results on this topics. In Section~\ref{sec:GMCcharp}, we discuss subtlety of Geometric Manin's conjecture in positive characteristic.

\

\noindent
{\bf Notation:} Let $k$ be a field. A variety $X$ defined over $k$ is an integral separated scheme of finite type over $k$. For any extension $k'/k$, we denote the base change of $X$ to $k'$ by $X_{k'}$. We denote the base change of $X$ to $\overline{k}$ by $\overline{X}$. For a scheme $X$, a component of $X$ means an irreducible component unless stated otherwise. 

Let $X$ be a smooth projective variety defined over $k$. We denote the numerical equivalence of Cartier divisors and $1$-cycles by $\equiv$. We denote the space of Cartier divisors on $X$ up to numerical equivalence by $N^1(X)_\bZ$ and the space of integral $1$-cycles on $X$ up to numerical equivalence by $N_1(X)_\bZ$. We denote $N^1(X)_\bZ \otimes \mathbb R$ by $N^1(X)$ and $N_1(X)_\bZ \otimes \mathbb R$ by $N_1(X)$, and these are finite dimensional real vector spaces. We denote the pseudo-effective cones of divisors and curves by $\overline{\mathrm{Eff}}^1(X) \subset N^1(X)$ and $\overline{\mathrm{Eff}}_1(X) \subset N_1(X)$ respectively. We also denote the nef cones of divisors and curves by $\Nef^1(X) \subset N^1(X)$ and $\Nef_1(X) \subset N_1(X)$ respectively too.

Let $X$ be a projective variety defined over $k$ and $L$ be a big and nef $\mathbb Q$-Cartier divisor on $X$. An $L$-line ($L$-conic, or $L$-cubic) is a birational map $f : C \to X$ to the image such that $C$ is a smooth geometrically integral and geometrically rational curve and $\deg f^*L = 1$ ($=2$ or $3$ respectively).

\bigskip

\noindent
{\bf Acknowledgements:}
The author would like to thank Roya Beheshti, Brian Lehamann, and Eric Riedl for our collaborations shaping my understanding of geometry of moduli spaces of rational curves. In particular the author would like to thank Brian for multiple collaborations on this subject. The author would also like to thank Brian for comments on an early draft of this paper.
The author would also like to thank the organizers of Miyako-no-Seihoku Algebraic Geometry Symposium 2021 for an opportunity to give a talk there.

Sho Tanimoto was partially supported by JSPS KAKENHI Early-Career Scientists Grant number 19K14512, by JSPS Bilateral Joint Research Projects Grant number JPJSBP120219935.

\section{Backgrounds}
\label{sec:background}

Let $k$ be a perfect field. 
Here we recall basic definitions and results on the geometry of Manin's conjecture.
We first introduce a birational invariant which plays a central role in Manin's conjecture:

\begin{defi}
Let $X$ be a smooth projective variety defined over $k$ and $L$ be a big and nef $\mathbb Q$-Cartier divisor on $X$. We define {\it the Fujita invariant} or {\it the $a$-invariant} as
\[
a(X, L) := \min\{ t\in \mathbb R \, | \, t L + K_X \in \overline{\mathrm{Eff}}^1(X)\},
\]
where $K_X$ is the canonical divisor of $X$ and $\overline{\mathrm{Eff}}^1(X)$ is the pseudo-effective cone of divisors on $X$. When $L$ is nef but not big, we set $a(X, L) = \infty$. When $k$ has characteristic $0$ and $L$ is big, it follows from \cite{BDPP} that $a(X, L) > 0$ if and only if $X$ is geometrically uniruled. When $k$ has positive characteristic, it follows from \cite[Theorem 1.6]{Das20} that $a(X, L) > 0$ if and only if $\overline{X}$ admits a dominant family of rational curves $C$ such that $K_X.C <0$.

When $X$ is singular but admits a smooth resolution $\beta: \widetilde{X} \to X$, we define the Fujita invariant as the Fujita invariant of the pullback of $L$ to a smooth resolution, i.e.,
\[
a(X, L) := a(\widetilde{X}, \beta^*L).
\]
This is well-defined due to the birational invariance of this invariant. (\cite[Propostion 2.7]{HTT15})
\end{defi}

\begin{exam}
Let $X$ be a smooth weak Fano variety defined over $k$, i.e., $X$ is projective and $-K_X$ is big and nef. Then we have $a(X, -K_X) = 1$.
\end{exam}

\begin{exam}
Let $X$ be a smooth projective variety defined over $k$ and $L$ be a big and nef $\mathbb Q$-divisor on $X$. Let $C \subset X$ be a geometrically integral curve such that $L.C >0$. Then $a(C, L|_C) >0$ if and only if $\overline{C}$ is rational. When $C$ is geometrically rational, we have
\[
a(C, L|_C) = \frac{2}{L.C}.
\]
\end{exam}

One of basic properties of the Fujita invariants is the following property:

\begin{prop}[{\cite[Corollary 4.5]{LST18}}]
Let $X$ be a geometrically integral smooth projective variety defined over $k$ and $L$ be a big and nef $\mathbb Q$-divisor on $X$. Then for any extension $k'/k$ we have
\[
a(X_{k'}, L_{k'}) = a(X, L).
\]
\end{prop}

One of main theorems regarding the $a$-invariants is the following theorem which claims that the exceptional set for the $a$-invariant in Manin's conjecture is a proper closed subset in characteristic $0$:

\begin{theo}[{\cite[Theorem 1.1]{HJ16} and \cite[Theorem 3.3]{LTGMC}}]
\label{theo:HJ}
Let $k$ be a field of characteristic $0$ and $X$ be a smooth geometrically uniruled projective variety defined over $k$. Let $L$ be a big and nef $\mathbb Q$-divisor on $X$. Then let $Y$ run over all subvarieties $Y$ defined over $k$ such that $a(Y, L|_Y) > a(X, L)$, the union
\[
V = \bigcup_Y Y \subset X
\]
is a proper closed subset of $X$.

\end{theo}

The proof of this theorem is relied on the recent solution for the boundedness of singular Fano varieties (BAB conjecture) by Birkar in \cite{birkar16b}.

\begin{exam}[{\cite[Lemma 4.1]{HTT15}}]
Let $S$ be a smooth del Pezzo surface defined over an algebraically closed field $k$ of characteristic $0$ with $L = -K_S$.
We have $a(S, L) = 1$ and for any rational curve $C$, we have $a(C, L|_C) > 1$ if and only if $-K_S.C = 1$. Hence $V$ in Theorem~\ref{theo:HJ} is the union of $-K_S$-lines.
\end{exam}

\begin{exam}[{\cite[Theorem 1.2]{BLRT20}}]
Let $k$ be an algebraically closed field of characteristic $0$ and let $X$ be a smooth Fano threefold with $L = -K_X$. Then it follows from \cite[Theorem 1.2]{BLRT20} that $V$ in Theorem~\ref{theo:HJ} is the union of $-K_X$-lines and exceptional divisors of divisorial contractions admitted by $X$.
\end{exam}

\

Next we introduce another invariant which also plays a central role in Manin's conjecture:

\begin{defi}
Let $X$ be a smooth projective variety defined over $k$ and $L$ be a big and nef $\mathbb Q$-Cartier divisor on $X$. We assume that $a(X, L) >0$. Then we define {\it the face associated to $L$} to be
\[
\mathcal F(k, X, L):= \Nef_1(X)\cap \{\alpha \in N_1(X) \, | \, (a(X, L)L + K_X).\alpha = 0\}.
\]
Note that since $(a(X, L)L + K_X).\alpha \geq 0$ for any $\alpha \in \Nef_1(X)$, this face is a usual face in the sense of convex geometry. We define {\it the $b$-invariant} as
\[
b(k, X, L) = \dim \mathcal F(k, X, L).
\]

When $X$ is singular, but admits a smooth resolution $\beta : \widetilde{X} \to X$, we define the $b$-invariant as the $b$-invariant of the pullback to a smooth resolution:
\[
b(k, X, L) = b(k, \widetilde{X}, \beta^*L).
\]
Again this is well-defined due to the birational invariance of this invariant. (\cite[Proposition 2.10]{HTT15})
\end{defi}

\begin{rema}
In the literature in Manin's conjecture, the $b$-invariant is defined as codimension of the minimal supported face of $\overline{\mathrm{Eff}}^1(X)$ containing $a(X, L)L +K_X$. These two definitions are equivalent due to the duality between $\Nef_1(X)$ and $\overline{\mathrm{Eff}}^1(X)$. See \cite[Definition 4.11]{LST18} for more details.
\end{rema}
\begin{rema}
The $a$-invariants and the $b$-invariants also play roles in the study of cylinders on Fano varieties. See, e.g., \cite[p. 11]{CPPZ20}.
\end{rema}

\begin{exam}
Let $X$ be a smooth weak Fano variety defined over $k$ with $L = -K_X$. Then $a(X, L)L + K_X \equiv 0$ so that we have $\mathcal F(k, X, L) = \Nef_1(X)$. Hence we have
\[
b(k, X, L) : = \dim N_1(X) = \dim N^1(X) =: \rho(X).
\]
\end{exam}

The following notion plays a role in the formulation of Geometric Manin's conjecture:

\begin{defi}
Let $X$ be a smooth projective variety defined over $k$ and $L$ be a big and nef $\mathbb Q$-divisor on $X$. Assume that $a(X, L) >0$. 

A dominant generically finite morphism $f: Y \to X$ from a smooth projective variety is called as {\it an $a$-cover} if $a(Y, f^*L) = a(X, L)$ holds.

Let $f : Y \to X$ be an $a$-cover. Then the pushforward map defines a homomorphism
\[
f_* : \mathcal F(k, Y, f^*L) \to \mathcal F(k, X, L).
\]
We say $f$ is {\it face contracting} if this homomorphism is not injective.
\end{defi}

\begin{rema}
When $b(k, Y, f^*L) > b(k, X, L)$, $f$ is automatically face contracting.
\end{rema}

\section{Manin's conjecture over finite fields and Batyrev's heuristic}
\label{sec:ManinBatyrev}

Let $k$ be a finite field and $X$ be a smooth weak Fano variety defined over $k$. Manin's conjecture over finite fields concerns an asymptotic formula for the counting function of $k$-rational curves on $X$. To this end we need to introduce the notion of moduli spaces of rational curves on $X$.

\subsection{Moduli spaces of rational curves}

Let $k$ be a field and $X$ be a smooth projective variety defined over $k$. One of candidates for the space of rational curves is the morphism scheme
\[
\mathrm{Mor}(\mathbb P^1, X),
\]
whose geometric closed points parametrize morphisms from $\mathbb P^1$ to $\overline{X}$. See \cite{Kollar} for a rigorous introduction to this scheme and \cite{Debarre} for a more friendly introduction. 
This scheme consists of at most countably many irreducible components. 

For each effective integral $1$-cycle $\alpha$,
\[
\mathrm{Mor}(\mathbb P^1, X, \alpha),
\]
denotes the morphism scheme whose geometric closed points are parametrizing $f : \mathbb P^1 \to \overline{X}$ such that $f_*\mathbb P^1 \equiv \alpha$. This is a quasi-projective scheme over $k$ and it consists of finitely many components. Let $U$ be a non-empty Zariski open subset of $X$. Then
\[
\mathrm{Mor}_U(\mathbb P^1, X, \alpha),
\]
denotes the open subscheme of $\mathrm{Mor}(\mathbb P^1, X, \alpha)$ parametrizing $f$ such that $f(\mathbb P^1)\cap U \neq \emptyset$.

For $f : \mathbb P^1 \to X$, the Zariski tangent space of $\mathrm{Mor}(\mathbb P^1, X)$ at $f$ is given by $H^0(\mathbb P^1, f^*T_X)$ and the obstruction space is given by $H^1(\mathbb P^1, f^*T_X)$ where $T_X$ is the tangent bundle of $X$. In particular, for any component $M \subset \mathrm{Mor}(\mathbb P^1, X, \alpha)$ we have
\[
\dim M \geq h^0(\mathbb P^1, f^*T_X) - h^1(\mathbb P^1, f^*T_X) = -K_X.\alpha + \dim X.
\]
Moreover, when $h^1(\mathbb P^1, f^*T_X) = 0$, the dimension of $\mathrm{Mor}(\mathbb P^1, X, \alpha)$ at $f$ is given by $-K_X. \alpha + \dim X$ and $\mathrm{Mor}(\mathbb P^1, X, \alpha)$ is smooth at $f$. We call this lower bound as {\it the expected dimension of $M$}.

\

Let $f : \mathbb P^1 \to X$ be a morphism. Then we have the decomposition
\[
f^*T_X = \mathcal O(a_1) \oplus \cdots \oplus \mathcal O(a_n)
\]
where $n = \dim X$ and $a_1 \geq \cdots \geq a_n$. We say $f$ is free if $a_n \geq 0$ and $f$ is very free if $a_n \geq 1$.

When $f$ is free, there exists a unique component $M \subset \mathrm{Mor}(\mathbb P^1, X)$ containing $f$. If we denote the universal family over $M$ by $\pi : \mathcal U \to M$, then the evaluation map $\mathrm{ev} : \mathcal U \to X$ is dominant and separable. Conversely if we have a component $M \subset \mathrm{Mor}(\mathbb P^1, X)$ such that for the universal family $\pi : \mathcal U \to M$, its evaluation map $\mathrm{ev} : \mathcal U \to X$ is dominant and separable, then a general member $f \in M$ is free.
In this situation we call $M$ as {\it a dominant and separable component}.

When $M$ is a component, we consider the following product of two evaluation maps
\[
\mathrm{ev}_2 : \mathcal U \times_M \mathcal U \to X \times X.
\]
Then $\mathrm{ev}_2$ is dominant and separable if and only if a general member of $M$ is very free.

\subsection{Batyrev's heuristic}

Let $k$ be a finite field with $\# k = q = p^n$.
Let $X$ be a smooth weak Fano variety defined over $k$.
Manin's conjecture over $k$ concerns an asymptotic formula for the counting function of rational curves of bounded degree on $X$. What does it mean?
Let $U$ be a Zariski open subset of $X$. We consider the following counting function:
\[
N(U, -K_X, d) = \sum_{\alpha \in \mathrm{Eff}_1(X)\cap N_1(X)_\bZ,\ -K_X.\alpha \leq d} \# \mathrm{Mor}_U(\mathbb P^1, X, \alpha)(k).
\]
This is the number of $k$-rational curves of anticanonical degree $\leq d$ meeting $U$ on $X$.
In his lecture notes (\cite{Bat88}), Batyrev described his heuristic argument to obtain an asymptotic formula of this counting function as $d \to \infty$. His heuristic was based on the following three assumptions:
\begin{enumerate}
\item There exists a Zariski open subset $U$ such that any component of $\mathrm{Mor}_U(\mathbb P^1, X, \alpha)$ is a dominant and separable component;
\item for each nef integral class of $1$-cycles $\alpha$, $\mathrm{Mor}_U(\mathbb P^1, X, \alpha)$ is irreducible;
\item and $\# \mathrm{Mor}_U(\mathbb P^1, X, \alpha)(k)$ is approximated by $q^{\dim \mathrm{Mor}_U(\mathbb P^1, X, \alpha)}$.
\end{enumerate}
Furthermore Ellenberg and Venkatesh suggested to justify (3) by using homological stability of $H^i_{c}(\mathrm{Mor}_U(\mathbb P^1, X, \alpha), \mathbb Q_\ell)$ combined with Grothedieck-Lefschetz trace formula.

The assumption (1) implies that for an effective integral class $\alpha$ which is not nef, $\mathrm{Mor}_U(\mathbb P^1, X, \alpha)$ is empty. For each nef class $\alpha$, the dimension of $\mathrm{Mor}_U(\mathbb P^1, X, \alpha)$ is given by the expected dimension $-K_X.\alpha + \dim X$. 
Thus the above counting function is roughly approximated by
\begin{align*}
N(U, -K_X, d) &= \sum_{\alpha \in \mathrm{Eff}_1(X)\cap N_1(X)_\bZ,\ -K_X.\alpha \leq d} \# \mathrm{Mor}_U(\mathbb P^1, X, \alpha)(k)\\
&= \sum_{\alpha \in \Nef_1(X)\cap N_1(X)_\bZ,\ -K_X.\alpha \leq d} \# \mathrm{Mor}_U(\mathbb P^1, X, \alpha)(k)\\
& \sim \sum_{\alpha \in \Nef_1(X)\cap N_1(X)_\bZ,\ -K_X.\alpha \leq d} q^{\dim \mathrm{Mor}_U(\mathbb P^1, X, \alpha)}\\
& = \sum_{\alpha \in \Nef_1(X)\cap N_1(X)_\bZ,\ -K_X.\alpha \leq d} q^{-K_X.\alpha + \dim X}\\
& \sim Cq^d d^{\rho(X)-1}, \quad \text{as $d \to \infty$ for some $C >0$},
\end{align*}
where the last asymptotic formula can be justified by a standard lattice point counting arguments.
The assumption (1) is true in characteristic $0$ in general, but not in characteristic $p$.
The assumption (2) fails even in characteristic $0$, and this failure is related to the existence of thin exceptional sets in Manin's conjecture over number fields. Geometric Manin's conjecture will serve as a replacement of the assumption (2) in characteristic $0$. Unfortunately the validity of the assumption (3) is not well-studied. There are some results on Cohen-Johns-Segal's conjecture on stability of homotopy type of the space of holomorphic spheres on homogeneous spaces (see \cite{CJS98}) and toric varieties (see \cite{Guest}) as well as a recent result on Fano hypersurfaces in \cite{BS20} inspired by the circle method which is a technique from analytic number theory.

\section{Geometric Manin's conjecture in characteristic $0$}
\label{sec:GMCchar0}

In this section, we assume that $k$ is an algebraically closed field of characteristic $0$. We discuss the validity of assumptions (1) and (2) in Batyrev's heuristic.

\subsection{Expected dimension}

First of all, the validity of the assumption (1) in Batyrev's heuristic is true in characteristic $0$. Here is the main theorem proved by Lehmann and the author:

\begin{theo}[{\cite[Theorem 1.1]{LTGMC}}]
\label{theo:expecteddimension}
Let $X$ be a smooth weak Fano variety defined over $k$ with $L = -K_X$. Let $V$ be the union of subvarieties $Y$ such that $a(Y, L|_Y) > a(X, L) = 1$. This is a proper closed subset by Theorem~\ref{theo:HJ}. Let $U$ be the complement of $V$. Then for any effective integral $1$-cycle $\alpha$, any component $M$ of $\mathrm{Mor}_U(\mathbb P^1, X, \alpha)$ is a dominant component. In particular $M$ generically parametrizes a free rational curve and we have $\dim M = -K_X.\alpha + \dim X$.
\end{theo}

Next we will discuss the assumption (2) and its replacement.

\subsection{Geometric Manin's conjecture}
The assumption (2) in Batyrev's heuristic fails in general. This is due to the existence of thin exceptional sets in Manin's conjecture documented in \cite{BT-cubic}, \cite{LeRu19}, and \cite{BHB20}.
Here is an example from \cite{LeRu19}:

\begin{exam}[{\cite{LeRu19}}]
Let $S = \mathbb P^1\times \mathbb P^1$ and $X = \mathrm{Hilb}^{[2]}(S)$ be the Hilbert scheme of length $2$ subschemes on $X$. Then $X$ is a smooth weak Fano variety. Let $L = -K_X$.
Then we have
\[
a(X, L) = 1, \quad b(X, L) = \dim N^1(X) = 3.
\]
Let $W$ be the blow up along the diagonal on $S \times S$.
Then we have a natural morphism $f : W \to X$. One can prove that
\[
a(W, f^*L) = 1, \quad b(W, f^*L) = 4.
\]
If we consider the pushforward map
\[
f_* : \mathcal F(W, f^*L)\cap N_1(W)_\bZ \to \mathcal F(X, L)\cap N_1(X)_\bZ.
\]
Then one can show that the image spans a rank $2$ lattice. Let $\alpha$ be an integral nef class on $X$ which is in the image of this map.
Then the size of $(f_*)^{-1}(\alpha)$ grows quadratically as the anticanonical degree of $\alpha$ grows.
For each $\beta \in  (f_*)^{-1}(\alpha)$, $\mathrm{Mor}(\mathbb P^1, W, \beta)$ is irreducible and a natural morphism
\[
\mathrm{Mor}(\mathbb P^1, W, \beta) \to \mathrm{Mor}(\mathbb P^1, X, \alpha)
\]
is a generically finite map to a component $M$ of $\mathrm{Mor}(\mathbb P^1, X, \alpha)$. 
Then since $\mathrm{Mor}(\mathbb P^1, W, \beta)$ and any component of $\mathrm{Mor}(\mathbb P^1, X, \alpha)$ are dominant components, we have
\[
\dim \mathrm{Mor}(\mathbb P^1, W, \beta) = -K_{W}.\beta + \dim W = -K_X.\alpha + \dim X = \dim M,
\]
because $\beta$ has vanishing intersection against $-f^*K_X + K_W$.
Thus $\mathrm{Mor}(\mathbb P^1, W, \beta)$ maps to $M$ dominantly.
Moreover a natural morphism
\[
\bigsqcup_{\beta \in (f_*)^{-1}(\alpha)}\mathrm{Mor}(\mathbb P^1, W, \beta) \to \mathrm{Mor}(\mathbb P^1, X, \alpha)
\]
is a degree $2$ map so that we conclude that $\mathrm{Mor}(\mathbb P^1, X, \alpha)$ admits many components as the anticanonical degree of $\alpha$ grows.
\end{exam}

Due to the existence of such examples, we proposed the following definitions in \cite{LTGMC} and \cite{BLRT20}:

\begin{defi}
Let $X$ be a smooth weak Fano variety defined over $k$ with $L = -K_X$.
Let $f : Y \to X$ be a generically finite morphism to the image from a smooth projective variety.
We say that $f$ is {\it a breaking map} if one of the following properties holds:
\begin{itemize}
\item $f.: Y \to X$ satisfies $a(Y, f^*L) > a(X, L)$;
\item $f$ is an $a$-cover and satisfies $\kappa(a(Y, f^*L)f^*L + K_Y) >0$ where $\kappa(D)$ is the Iitaka dimension of a $\mathbb Q$-divisor $D$, or;
\item $f$ is an $a$-cover
and $f$ is face contracting.
\end{itemize}

We say a component $M \subset \mathrm{Mor}(\mathbb P^1, X)$ is {\it an accumulating component} if there exists a breaking map $f : Y \to X$ and a component $N$ of $\mathrm{Mor}(\mathbb P^1, Y)$ such that $f$ induces a generically finite and dominant rational map $N \dashrightarrow M$. We say $M$ is {\it a Manin component} if $M$ is not an accumulating component.
\end{defi}

\begin{rema}
The above notion of Manin components is slightly different from the one in \cite{BLRT20}. It is closer to the notion of {\it good components} in \cite{BLRT20}.
\end{rema}

Here is one of formulations of Geometric Manin's conjecture:
\begin{conj}[Geometric Manin's conjecture ({\cite{LTGMC} and \cite{BLRT20}})]
Let $X$ be a smooth weak Fano variety defined over $k$ with $L = -K_X$. Then there exist $\tau \in \Nef_1(X)\cap N_1(X)_\bZ$ and a positive integer $n$ such that for any integral nef class $\alpha \in \tau + \Nef_1(X)$, $\mathrm{Mor}(\mathbb P^1, X, \alpha)$ contains exactly $n$ Manin components.
\end{conj}

In the next subsection, we will discuss what is known about this conjecture.

\subsection{Examples of Geometric Manin's conjecture}

The irreducibility of moduli spaces of rational curves (GMC) is known in the following cases:
\begin{itemize}
\item Fano hypersurfaces (\cite{HRS04}, \cite{BK13}, \cite{RY19}, and \cite{BV17});
\item homogeneous spaces (\cite{Thom98} and \cite{KP01});
\item toric varieties (\cite{Bou12} and \cite{Bou16});
\item del Pezzo surfaces (\cite{Testa05}, \cite{Testa09}, and \cite{BLRT21});
\item moduli spaces of vector bundles (\cite{Cast04} and \cite{MB20})
\item smooth Fano threefolds (\cite{Starr00}, \cite{Cast04}, \cite{LTGMC}, \cite{LTJAG}, \cite{BLRT20}, and \cite{ST21});
\item del Pezzo fibrations (\cite{LT19} and \cite{LTGT}).
\end{itemize}

Many of the above results are based on the method pioneered in \cite{HRS04} which used two main tools in algebraic geometry, i.e., the moduli space of stable maps of genus $0$ with $n$ marked points $\overline{M}_{0,n}(X)$ and Mori's Bend and Break lemma. (See \cite{FP97} for an introduction to the moduli spaces of stable maps and \cite{KM98} for Bend and Break lemmas.) 

An idea is that one can prove the irreducibility of moduli spaces of rational curves of degree $d$ by using the induction on $d$ combined with Bend and Break lemma in the moduli space of stable maps.
This strategy has been sharpened in \cite{BLRT20} for smooth Fano threefolds, and here are main results from this paper: 

The first result is called as {\it Movable Bend and Break} which is an essential ingredient for the inductive steps of various inductive proofs of GMC, and this is an improvement of Mori's Bend and Break lemma:

\begin{theo}[Movable Bend and Break ({\cite[Theorem 1.4]{BLRT20}})]
\label{theo:MBB}
Let $X$ be a smooth Fano threefold and $M$ be a component of $\overline{M}_{0,0}(X)$ generically parametrizing free rational curves. Suppose that a general member $C$ of $M$ has anticanonical degree $\geq 6$. Then $M$ contains a stable map of the form $f : Z \to X$ such that $Z$ consists of two components and the restriction of $f$ to each component realizes this curve as a free rational curve. 
\end{theo}

The next result is the irreducibility of a general fiber of the evaluation map and the result is based on the classification of $a$-covers for smooth Fano threefolds (\cite[Lemma 5.2, Theorem 5.3, and Theorem 5.4]{BLRT20}):

\begin{theo}[{\cite[Theorem 1.3]{BLRT20}}]
\label{theo:a-covers}
Let $X$ be a smooth Fano threefold and $M$ be a component of $\overline{M}_{0,0}(X)$ generically parametrizing free rational curves. Let $\mathcal C \to M$ be the corresponding component of $\overline{M}_{0,1}(X)$ above $M$. Suppose that a general fiber of the evaluation map $\mathrm{ev} : \mathcal C \to X$ is not irreducible. Then either
\begin{itemize}
\item $M$ generically parametrizes stable maps whose images are $-K_X$-conics, or;
\item $M$ parametrizes a family of curves contracted by a del Pezzo fibration $\pi : X \to Z$.
\end{itemize}
\end{theo}

These theorems were used in the proofs of Geometric Manin's conjecture for the following Fano threefolds:
\begin{itemize}
\item del Pezzo threefolds of degree $\geq 2$ (\cite{Starr00}, \cite{Cast04}, and \cite[Section 7]{LTGMC});
\item most prime Fano threefolds of index $1$ and Picard rank $1$ (\cite{LTJAG}); 
\item arbitrary smooth quartic threefolds and the unique Fano threefolds with two $E_5$-type divisorial contractions (\cite[Section 8]{BLRT20}), and;
\item general del Pezzo threefolds of degree $1$(\cite{ST21}). 
\end{itemize}
Most proofs of the above results are based on the inductive proofs using Movable Bend and Break.
The inductive steps can be justified by using Theorems~\ref{theo:MBB} and \ref{theo:a-covers}, and the remaining tasks are the analysis of the base cases, i.e., the irreducibility of the moduli spaces of $-K_X$-conics, cubics, quartics, and so on.

There is a complete classification of smooth Fano threefolds by Iskovskih and Mori--Mukai (\cite{IskovI}, \cite{IskovII}, \cite{MM81}, and \cite{MM03}), so it is natural to propose the following problem:

\begin{prob}
Prove Geometric Manin's conjecture for all smooth Fano threefolds following the classification of Iskovskih and Mori--Mukai.
\end{prob}

We believe that Theorems~\ref{theo:MBB} and \ref{theo:a-covers} would be useful to attack this problem.
We also expect that Movable Bend and Break holds in higher dimension too:

\begin{conj}[Movable Bend and Break conjecture]
Let $X$ be a smooth Fano variety of dimension $n$ defined over $k$. Then there exists a positive integer $c(n)$ such that if $M$ is a component of $\overline{M}_{0,0}(X)$ generically parametrizing free rational curves of anticanonical degree $\geq c(n)$, then $M$ contains a stable map of the form $f : Z \to X$ such that $Z$ consists of two components and the restriction of $f$ to each component realizes this curve as a free rational curve. 
\end{conj}

The above conjecture combined with \cite[Theorem 1.4]{LTGMC} will imply Batyrev's conjecture predicting a polynomial growth of the number of dominant components. See \cite[Section 5]{LTGMC} for more details.

\section{Geometric Manin's conjecture in characteristic $p$}
\label{sec:GMCcharp}

Let $k$ be an algebraically closed field of characteristic $p >0$. Here we discuss Geometric Manin's conjecture over $k$ and its subtlety in characteristic $p$. At this moment we do not fully understand how one should formulate Geometric Manin's conjecture in characteristic $p$ and Manin's conjecture over finite fields in general.

Before we start our main discussion, let us mention a few facts about the moduli space of stable maps of genus $0$. Let $X$ be a smooth projective variety defined over $k$. Let $\alpha$ be an integral effective $1$-cycle on $X$ and we denote the coarse moduli space of stable maps of genus $0$ and class $\alpha$ by
\[
\overline{M}_{0,0}(X, \alpha).
\]
Let $M$ be a component of this moduli space. Then one has
\[
\dim M \geq -K_X.\alpha + \dim X -3.
\]
We call this lower bound as the expected dimension of $M$. Note that this expected dimension is different from the expected dimension of $\mathrm{Mor}(\mathbb P^1, X, \alpha)$. This is due to the presence of the action of $\mathrm{PGL}_2$ on $\mathrm{Mor}(\mathbb P^1, X, \alpha)$. When $M$ generically parameterizes a birational stable map to a free rational curve, the dimension of $M$ is equal to the expected dimension and this stable map is a smooth point of the moduli space.

We denote the union of components of $\overline{M}_{0,0}(X)$ generically parametrizing stable maps from $\mathbb P^1$ by $\overline{\mathrm{Rat}}(X)$. We also define $\overline{\mathrm{Rat}}(X, \alpha)$ for an effective integral $1$-cycle $\alpha$ in a similar way.

\subsection{Weak Manin's conjecture}

Let us start our main discussion by stating weak Manin's conjecture over number fields:

\begin{conj}[Weak Manin's conjecture]
\label{conj:weakManin}
Let $F$ be a number field and $X$ be a smooth geometrically uniruled projective variety defined over $F$. Let $L$ be a big and nef $\mathbb Q$-divisor on $X$.
We consider a height function associated to $L$:
\[
\mathsf H_L : X(F) \to \mathbb R_{\geq 0}.
\]
Then there exists a non-empty Zariski open subset $U \subset X$ such that 
if we define the counting function as
\[
N(U, L, T): = \# \{ P \in U(F) \, | \, \mathsf H_L(P) \leq T \},
\]
then for any $\epsilon > 0$, we have
\[
N(U, L, T) = O(T^{a(X, L) + \epsilon}).
\]
\end{conj}

The following proposition and Theorem~\ref{theo:HJ} are consistent with this conjecture:

\begin{prop}
\label{prop:ramification}
Let $X$ be a smooth projective variety over $k$ and $L$ be a big and nef $\mathbb Q$-divisor on $X$. Let $f : Y \to X$ be a separable generically finite and dominant morphism from a smooth projective variety. Then we have $a(Y, f^*L) \leq a(X, L)$.
\end{prop}

\begin{proof}
Since $f$ is separable, we have the ramification formula:
\[
K_Y = f^*K_X + R
\]
with an effective divisor $R \geq 0$ on $Y$.
Thus 
\[
a(X, L)f^*L + K_Y = f^*(a(X, L)L + K_X) + R
\]
is pseudo-effective so that $a(Y, f^*L) \leq a(X, L)$.
\end{proof}

It turns out that Theorem~\ref{theo:HJ} and the above proposition fail in positive characteristic:

\begin{exam}[{\cite[Table 2]{KN20b} and \cite[Example 1.13]{BLRT21}}]
Assume that our ground field is $\mathbb F_3$.
Let $S'$ be the surface in $\mathbb P(1, 1, 2, 3)_{(x:y:z:w)}$ defined by
\[
w^2 + z^3 = x^2y^2(x+y)^2.
\]
This is a du Val del Pezzo surface of degree $1$ (\cite[Talbe 2]{KN20b}).
Let $S \to S'$ be the minimal resolution so that $S$ is a weak del Pezzo surface of degree $1$.

Let $\beta : \widetilde{S} \to S$ be the blow up of the base point of $|-K_S|$. Then $|-K_{\widetilde{S}}|$ defines a quasi-elliptic fibration $\pi : \widetilde{S} \to B = \mathbb P^1$, i.e., a general geometric fiber $C$ is a rational curve of arithmetic genus $1$. Then we have $a(C, -\beta^*K_S|_C) = 2 > 1 = a(S, -K_S)$, thus Theorem~\ref{theo:HJ} fails in dimension $2$ and characteristic $3$.
Let $F : B' = \mathbb P^1 \to \mathbb P^1 = B$ be the Frobenius map and we consider the base change $Y = S \times_B B'$. We denote its normalization by $\widetilde{Y} \to Y$ and consider $\rho : \widetilde{Y} \to B'$ which is a generically smooth fibration. Then we have a purely inseparable map $f : \widetilde{Y} \to S$.

Now we take the base change to $K = \mathbb F_3(t)$.
Then $B'(K)$ is Zariski dense in $B'_K$ and a general rational fiber $C_K$ of $\rho$ is isomorphic to $\mathbb P^1_K$ and we have $a(C_K, -f^*K_S) = 2$. So the asymptotic formula for the counting function of raitonal points on $C_K$ is of the form $q^{2d}$, hence Conjecture~\ref{conj:weakManin} fails in positive characteristic.

\end{exam}

\begin{exam}[{\cite{KM99}, \cite{CT18}, and \cite[Example 1.14]{BLRT21}}]
We work over $\mathbb F_2$. The following surface was found in \cite{KM99} and studied in \cite{CT18} due to its failure of Kawamata-Viehweg vanishing theorem.

Let $S$ be the blow up of $\mathbb P^2$ at all seven $\mathbb F_2$-points. Then $S$ is a weak del Pezzo surface of degree $2$. All $(-2)$-curves on $S$ are the strict transforms of $\mathbb F_2$-lines on $\mathbb P^2$. Let $S'$ be the anticanonical model of $S$. Then $S'$ comes with seven $A_1$ singularities. Moreover $|-K_S|$ defines a purely inseparable double cover $\Phi_{|-K_S|}:S \to \mathbb P^2$.

We write the equation of $S'$ in $\mathbb P(1, 1, 1, 2)_{(x, y, z, w)}$ by
\[
w^2 = f_4(x, y, z),
\]
where $f_4$ is a homogeneous polynomial of degree $4$ with coefficients in $\mathbb F_2$. We define the morphism
\[
f : \mathbb P^2 \to S', (s:t:u) \mapsto (s^2:t^2:u^2:f_4(s, t, u)).
\]
Then the Frobenius map $F : \mathbb P^2 \to \mathbb P^2$ factors through $f$.
Let $Y$ be a resolution of $\mathbb P^2$ so that the rational map $f': Y \dashrightarrow S$ is a morphism.
Then we have
\[
a(Y, -f^*K_S) = \frac{3}{2} > 1 = a(S, -K_S).
\]
Thus Proposition~\ref{prop:ramification} may fail if the morphism is not separable. 
By working over $K = \mathbb F_2(t)$, Conjecture~\ref{conj:weakManin} fails for this example as well.
\end{exam}

For a weak del Pezzo surface $S$ over $k$, we have a complete classification of inseparable covers $f: Y \to S$ such that $a(Y, -f^*K_S) > a(S, -K_S)$ (\cite[Theorem 1.12]{BLRT21}). As a byproduct we have the following theorem:

\begin{theo}[{\cite{BLRT21}}]
\label{theo:breakingmaps}
Let $S$ be a smooth weak del Pezzo surface defined over $k$. Then there exists an inseparable cover $f : Y \to S$ such that $a(Y, -f^*K_S) > a(S, -K_S)$ if and only if a general member of $|-K_S|$ is singular.
In particular such a cover does not exist when $S$ is a smooth del Pezzo surface.
\end{theo}

The classification of these pathological weak del Pezzo surfaces is given in \cite{KN20a} and \cite{KN20b}. The motivation of these papers is the classification of du Val del Pezzo surfaces which do not satisfy Kawamata-Viehweg vanishing or log liftability. It is interesting to see how various pathologies are related.

\subsection{Expected dimension and separability: the case of surfaces}

Let us consider these pathological weal del Pezzo surfaces $S$ such that a general member of $|-K_S|$ is singular. First of all this only happens for weak del Pezzo surfaces of degree $\leq 2$.
When $S$ has degree $1$, a general member of $|-K_S|$ is rational so we have a $1$-dimensional family of $-K_S$-lines whose dimension is higher than expected dimension. When $S$ has degree $2$, again a general member of $|-K_S|$ is rational so we have a $2$-dimensional family of $-K_S$-conics
whose dimension is higher than expected dimension. Conversely if $S$ is a weak del Pezzo surface such that a general member of $|-K_S|$ is smooth, then one can prove that all families of $-K_S$-lines and conics have expected dimension. By combining this result with inductive arguments using Bend and Break, one can prove the following theorem:

\begin{theo}[{\cite[Theorem 1.1]{BLRT21}}]
Let $S$ be a weak del Pezzo surface defined over $k$.
Then the following statements are equivalent:
\begin{itemize}
\item a general member of $|-K_S|$ is smooth;
\item all dominant components of $\overline{\mathrm{Rat}}(S)$ has expected dimension.
\end{itemize}
In particular when $S$ is a del Pezzo surface, all dominant components of $\overline{\mathrm{Rat}}(S)$ has expected dimension.
\end{theo}

\begin{rema}
According to Theorem~\ref{theo:breakingmaps}, the above statements are also equivalent to the existence of inseparable covers $f : Y \to S$ such that $a(Y, -f^*L_S) > a(S, -K_S)$. Moreover one can show that all families of $-K_S$-lines and conics whose dimensions are higher than expected dimensions will lift to some of these inseparable covers. In the view of Theorem~\ref{theo:expecteddimension} we expect that for every dominant component $M$ of $\overline{\mathrm{Rat}}(S)$ whose dimension is higher than expected dimension,  its evaluation map $\mathrm{ev} : \mathcal C_M \to S$ from the universal family will factor through an inseparable cover $f: Y \to S$ with $a(Y, -f^*K_S) > a(S, -K_S)$ at least after taking a base change of $\pi_M: \mathcal C_M \to M$. However, such an expectation may not hold in higher dimension.
\end{rema}

Next let us focus on a del Pezzo surface $S$. Even though every dominant component of $\overline{\mathrm{Rat}}(S)$ has expected dimension, it is still possible that some of them are not separable.
\begin{exam}
Assume that $k$ has characteristic $2$.
Let $S$ be the Fermat cubic surface defined by
\[
x^3+ y^3 + z^3 + w^3 = 0.
\]
Then one can show that a general rational member of $|-K_S|$ is a cuspidal rational curve $C$. From this one can show that if we denote the normalization by $f : \mathbb P^1 \to C \hookrightarrow S$, then $f^*T_S = \mathcal O(4)\oplus \mathcal O(-1)$ so that it is not free.
\end{exam}

We completely determined which del Pezzo surfaces of degree $\geq 2$ admit such inseparable families in \cite{BLRT21}:
\begin{theo}[{\cite[Theorem 1.2]{BLRT21}}]
\label{theo:separability}
Let $S$ be a smooth del Pezzo surface of degree $d$ defined over $k$ with $\mathrm{char}(k) = p$. Assume that either $d \geq 2$, or $d= 1$ and $p\geq 11$. When $d = 3$ we assume that $S$ is not isomorphic to the following surface:
\begin{itemize}
\item the Fermat cubic surface in characteristic $2$.
\end{itemize}
When $d = 2$ we assume that $S$ is not isomorphic to the following surfaces:
\begin{itemize}
\item the double cover of $\mathbb P^2$ ramified along the Fermat quartic curve in characteristic $3$;
\item the double cover of $\mathbb P^2$ ramified along the double line in characteristic $2$.
\end{itemize}
Then every dominant component of $\overline{\mathrm{Rat}}(S)$ is separable so it generically parametrizes a free rational curve.

Moreover if $S$ has degree $\geq 2$ and $S$ is isomorphic to one of the above exceptions, then there exists an inseparable dominant component of $\overline{\mathrm{Rat}}(S)$.
\end{theo}

The proof is based on inductive arguments again. First we show the above statement for $-K_S$-conics and cubics. Then we apply Bend and Break arguments to obtain separability for higher degree rational curves.

One can show that the above set of exceptional del Pezzo surfaces of degree $\geq 2$ exactly coincides with the set of non-$F$-split del Pezzo surfaces of degree $\geq 2$. 
Thus it would be nice to produce a direct proof of the following statement: if a smooth del Pezzo surface is $F$-split, then every dominant component of $\overline{\mathrm{Rat}}(S)$ is separable so it generically parametrizes a free rational curve. A solution to this may improve results for del Pezzo surfaces of degree $1$.

\subsection{Geometric Manin's conjecture for del Pezzo surfaces}
Finally let us mention our results on the irreducibility of moduli spaces of rational curves on del Pezzo surfaces.
Let us introduce the following function:
\[
\delta(d) = 
\begin{cases}
2 & d \geq 4\\
3 & d = 2, 3\\
11 & d = 1.
\end{cases}
\]
Here is our main theorem in \cite{BLRT21}:

\begin{theo}[{\cite[Theorem 1.5]{BLRT21}}]
Let $S$ be a smooth del Pezzo surface of degree $d$ defined over an algebraically closed field of characteristic $p$. Assume that $p \geq \delta(d)$. Furthermore when $d = 2$, we assume that $S$ is not one of exceptions listed in Theorem~\ref{theo:separability}.

Let $\alpha$ be an integral nef class on $S$ such that $-K_S.\alpha \geq 3$. Then:
\begin{itemize}
\item if $\alpha$ is not a multiple of $-K_S$-conics, then $\overline{\mathrm{Rat}}(S, \alpha)$ is irreducible and it generically parametrizes a birational stable map to a free rational curve;
\item if $\alpha$ is a multiple of a smooth rational conic, then $\overline{\mathrm{Rat}}(S, \alpha)$ is irreducible and it generically parametrizes a stable map to a smooth rational conic;
\item if $d = 2$ and $\alpha$ is a multiple of $-K_S$ or $d = 1$ and there is a contraction $\beta: S \to S'$ of a $(-1)$-curve such that $\alpha$ is a multiple of $-\beta^*K_{S'}$, then $\overline{\mathrm{Rat}}(S, \alpha)$ consists of two components: one generically parametrizes a birational stable map and the other generically parametrizes a stable map to a $-K_S$-conic, and;
\item if $d = 1$ and $\alpha$ is a multiple of $-2K_S$, then $\overline{\mathrm{Rat}}(S, \alpha)$ consists of at least two components: there is a unique component generically parametrizing a birational stable map and the others generically parametrize a stable map to a $-K_S$-conic.
\end{itemize}
\end{theo}

The proof of this theorem is based on inductive arguments as well as lifting arguments to characteristic $0$. More precisely we denote the union of components of $\overline{\mathrm{Rat}}(S, \alpha)$ generically parametrizing birational stable maps by $\overline{M}^{\mathrm{bir}}(S, \alpha)$. Let $-K_S.\alpha = e \geq 3$. We fix $e-2$ general points and we prove that the loci parametrizing stable maps in $\overline{M}^{\mathrm{bir}}(S, \alpha)$ passing through these $e-2$ general points are $1$-dimensional and contained in the smooth locus of $\overline{M}_{0,0}(S)$ using the inductive proof on $e$. Then we lift $S$ to characteristic $0$ and prove that these loci in characteristic $0$ are connected by appealing to the irreducibility of the space of rational curves on del Pezzo surfaces in characteristic $0$ proved by Testa (\cite{Testa05} and \cite{Testa09}). Then using a specialization argument we prove that these loci in characteristic $p$ are connected as well. Since this locus is connected and contained in the smooth locus of the entire moduli space, we conclude that $\overline{M}^{\mathrm{bir}}(S, \alpha)$ is irreducible.

\nocite{*}
\bibliographystyle{alpha}
\bibliography{absurvey}

\end{document}